\documentclass{article}
\usepackage{graphicx}
\usepackage[utf8]{inputenc}
\usepackage[T1]{fontenc}
\usepackage{amsmath, amssymb, amsthm, amsfonts}
\usepackage{enumitem}
\usepackage{mathrsfs}
\usepackage{hyperref}

\theoremstyle{plain}
\newtheorem{theorem}{Theorem}[section]
\newtheorem{lemma}[theorem]{Lemma}

\newtheorem{claim}[theorem]{Claim}

\theoremstyle{definition}
\newtheorem{definition}[theorem]{Definition}

\newtheorem{question}[theorem]{Question}

\newcommand{\forces}{\Vdash}
\newcommand{\forceP}{\mathbb{P}}
\newcommand{\forceQ}{\mathbb{Q}}

\newcommand{\code}{\operatorname{Code}}
\newcommand{\Lsharp}{L^{\#}}

\title{On \texorpdfstring{$\mathbf{\Sigma}^1_3$}{Sigma13}- and \texorpdfstring{$\Sigma^1_4$}{Sigma14}-uniformization}
\author{Stefan Hoffelner \footnote{The author's research was funded in whole by the Austrian Science Fund (FWF) Grant-DOI 10.55776/P37228. For the purpose of open access, the authors have applied a CC BY public copyright license to any Author Accepted Manuscript version arising from this submission.}}
\date{
    TU Wien \\
    March 2026
}

\begin{document}

\maketitle

\begin{abstract}
    Assuming the consistency of $\mathsf{ZFC}$, we construct a model of set theory in which the boldface $\mathbf{\Sigma}^1_3$-uniformization property holds, yet the lightface $\Sigma^1_4$-uniformization property fails, separating these two principles for the first time. We also indicate how to create a universe where $\Sigma^1_3$-uniformization holds, but $\Sigma^1_4$-uniformization fails using inner models with large cardinals. 
\end{abstract}

\section{Introduction}

The uniformization problem---asking whether every relation in a given pointclass contains a function with the same domain in the same pointclass---has been a central driving force in descriptive set theory since its inception around 1930. While the Axiom of Choice guarantees the existence of uniformizing functions in the abstract, the definability of such functions is a much more delicate matter. 

At the foundational levels of the projective hierarchy, the situation is entirely resolved within standard set theory. A classical result of Kondo \cite{Kondo1939} establishes that every $\mathbf{\Pi}^1_1$ relation can be uniformized by a $\mathbf{\Pi}^1_1$ function. This readily implies that the $\mathbf{\Sigma}^1_2$-uniformization property also holds in $\mathsf{ZFC}$. 

However, for pointclasses beyond the second level of the projective hierarchy, $\mathsf{ZFC}$ alone is insufficient to decide uniformization. The behavior of uniformization at the third level and above heavily depends on the underlying structural properties of the universe, historically dividing into two mutually exclusive paradigms: the existence of good definable wellorders of the reals and the assumption of strong determinacy hypotheses.

In canonical inner models, such as G\"{o}del's constructible universe $L$ or larger core models, the reals admit a good projective wellordering. For instance, in $L$, the canonical wellordering of the reals is $\Delta^1_2$, which  implies that the $\mathbf{\Sigma}^1_n$-uniformization property holds for all $n \geq 2$. More generally, if a model admits a $\Delta^1_n$ good wellordering of the reals, then by an observation of J. Addison $\mathbf{\Sigma}^1_k$-uniformization has to be true for all $k \geq n$. In all prior known models where the $\mathbf{\Sigma}^1_3$-uniformization property was established, it was achieved via the presence of a good projective wellorder that inextricably carried the ${\Sigma}^1_4$-uniformization property along with it. 

Conversely, under Projective Determinacy ($\mathsf{PD}$), or equivalently, as famously shown by D. Martin, J. Steel and H. Woodin (\cite{MS},\cite{muller2020mice}), under sufficiently large cardinals the picture is radically different. The periodicity theorems of Y. Moschovakis \cite{moschovakis2025descriptive} imply that uniformization alternates: at the third level, it is the $\mathbf{\Pi}^1_3$ relations that uniformize, while $\mathbf{\Sigma}^1_3$-uniformization explicitly fails. At the fourth level, $\mathbf{\Sigma}^1_4$-uniformization holds and so on. 

This historical dichotomy has left open a natural structural question regarding the independence of these projective levels from one another. Specifically, must the $\mathbf{\Sigma}^1_3$-uniformization property inevitably entail ${\Sigma}^1_4$-uniformization?

In this paper, we answer this question in the negative.  Our main theorem demonstrates that it is possible to  preserve boldface uniformization at the third level while simultaneously arranging for its failure at the lightface fourth, all without the assumption of large cardinals.

The required construction is considerably more delicate than one might initially expect, a difficulty anticipated by recent developments in the field. In \cite{HOFFELNER2025110272}, a flexible coding technique is introduced that forces the $\Sigma^1_{n+2}$-uniformization property for all $n \in \omega$ simultaneously. The robustness of this technique presents a significant obstacle, effectively ruling out straightforward approaches to our main question. Indeed, any naïve attempt to separate these levels would likely be compatible with the methods of \cite{HOFFELNER2025110272}, which would inadvertently force $\Sigma^1_4$-uniformization to hold and thus invalidate the separation. Consequently, we are compelled to develop a somewhat non-linear forcing iteration---one that eschews a strict chronological progression in favor of necessary rollbacks to earlier stages. Our construction implements these non-standard features.

\begin{theorem}
    There exists a generic extension of $L$ in which the boldface $\mathbf{\Sigma}^1_3$-uniformization property holds, yet the lightface $\Sigma^1_4$-uniformization property fails.
\end{theorem}

This result separates these two properties at the projective levels, thereby producing a behaviour distinct from both standard inner models and determinacy hypotheses.

The proof of the above theorem uses arguments which seem to be insufficient to separate the lightface $\Sigma^1_3$-uniformization from $\Sigma^1_4$-uniformization for reasons which we shall discuss. A separation of $\Sigma^1_3$-uniformization from $\Sigma^1_4$-uniformization can be achieved nevertheless if we work over $L^{\#}$, the least inner model where every set has a sharp as the ground model. The argument is utilizing the well-known argument of A. L\'evy (see \cite{levy1965definability}) and the fact that the good wellorder of $L^{\#}$ and its generic absoluteness suffice to preserve $\Sigma^1_3$-uniformization in generic extensions of $L^{\#}$.

\begin{theorem}
    Working over $L^{\#}$ there is a generic extension where $\Sigma^1_3$ uniformization is true, yet $\Sigma^1_4$ uniformization fails. In fact there is an $\Pi^1_3$-set which can not be uniformized by an ordinal definable function.
\end{theorem}

\section{The Ground Model and the Coding Machinery}
To establish our separation result, we rely heavily on the foundational coding machinery developed in \cite{hoffelner2023forcing}. We use the considerably shorter version from \cite{hoffelner2026failurepi1n3reductionpresencesigma1n3separation}. In this section, we provide an exposition of the ground model $W$, the coding mechanism via independent Suslin trees, and the base framework of allowable forcings. 

\subsection{Independent Suslin Trees and Almost Disjoint Coding}
Our construction relies on the ability to generically destroy Suslin trees independently of one another.

\begin{definition}
Let $\vec{T} = (T_\alpha \mid \alpha < \kappa)$ be a sequence of Suslin trees. We say that the sequence is an \emph{independent family of Suslin trees} if for every finite set of pairwise distinct indices $e = \{e_0, e_1, \dots, e_n\} \subset \kappa$, the product $T_{e_0} \times T_{e_1} \times \dots \times T_{e_n}$ is a Suslin tree.
\end{definition}

In $L$, we can canonically define an $\omega_1$-sequence of such trees. In fact $\diamondsuit$ implies the existence of $2^{\aleph_1}$-many such Suslin trees (\cite{zakrzewski1981weak}). We just need $\aleph_1$-many of them, which however should be easily definable. The following lemma is proved in (\cite{hoffelner2026failurepi1n3reductionpresencesigma1n3separation} and also in \cite{hoffelner2025universelargecontinuumglobal}).

\begin{theorem} \label{DefinitionIndependentSequence}
Assume that $\aleph_1 = \aleph_1^L$ and that $(M, \in)$ is a transitive, $\omega_1$-containing, uncountable model of $\mathsf{ZFC}^- + \text{``}\aleph_1 \text{ exists''}$. Then there is an independent sequence $\vec{T} = (T_\alpha \mid \alpha < \omega_1)$ of $L$-Suslin trees, and the sequence $\vec{T}$ is uniformly $\Sigma_1(\{\omega_1\})$-definable over $M$. Specifically, there is a $\Sigma_1$-formula $\phi$ with $\omega_1$ as the unique parameter such that the relation $\{(t, \gamma, \eta) \mid \gamma, \eta < \omega_1 \land t \in T_\eta^\gamma\}$ is definable over $M$ using $\phi$, where $T_\eta^\gamma$ denotes the $\gamma$-th level of $T_\eta$.
\end{theorem}

We split $\vec{T}$ into two independent sequences, which will be utilized to separate our projective properties:
\[ \vec{S}^0 := \{ S_\alpha \mid \alpha \text{ is even}\} \quad \text{and} \quad \vec{S}^1 := \{ S_\alpha \mid \alpha \text{ is odd}\}. \]

To encode these branches into reals, we employ the almost disjoint coding forcing $\mathbb{A}_D(X)$. We fix a canonical $L$-definable almost disjoint family $D = \{d_\alpha \mid \alpha < \aleph_1\} \subset [\omega]^\omega$. For a set $X \subset \omega_1$, $\mathbb{A}_D(X)$ is the ccc forcing that adds a real $x$ such that $\alpha \in X$ if and only if $x \cap d_\alpha$ is finite. As $\mathbb{A}_D(X)$ has the Knaster property, it preserves the Suslinity of any Suslin tree.

\subsection{The Ground Model and the Coding Forcing}
We first force over $L$ to destroy all members of $\vec{S} = \vec{S}^0 \cup \vec{S}^1$ via generically adding an $\omega_1$-branch. We form the finite support product:
\[ \mathbb{P}^0 := \prod_{\alpha < \omega_1} S_\alpha. \]
Note that this is an $\aleph_1$-sized, ccc forcing over $L$. In a second step, we add $\omega_1$-many Cohen subsets of $\omega_1$ using a countably supported product evaluated in $L$: 
\[ \mathbb{P}^1 := \Big(\prod_{\alpha < \omega_1} \mathbb{C}(\omega_1)\Big)^L. \]

Note that $\mathbb{P}^1$ is $\sigma$-closed only over $L$, so we need to ensure that the two-step iteration $\mathbb{P}^0 \ast \mathbb{P}^1$ preserves $\aleph_1$. Forcing with the two-step iteration is equivalent to forcing with the product $\mathbb{P}^0 \times \mathbb{P}^1$. Because $\mathbb{P}^0 \times \mathbb{P}^1$ is isomorphic to $\mathbb{P}^1 \times \mathbb{P}^0$, we can view the extension as first forcing with $\mathbb{P}^1$. Since $\mathbb{P}^1$ is $\sigma$-closed over $L$, it does not add new $\omega_1$-branches to trees from $L$. Thus, $\vec{S}$ remains an independent sequence of Suslin trees in $L^{\mathbb{P}^1}$, and $\mathbb{P}^0$ remains a ccc forcing over $L^{\mathbb{P}^1}$. So $\aleph_1$ is preserved and $\mathsf{CH}$ remains true. 

Our ground model is defined as:
\[ W := L[G^0][G^1], \]
where $G^0 \times G^1$ is a generic filter for $\mathbb{P}^0 \times \mathbb{P}^1$.

Let $t \in W$ be a real (or a recursively coded tuple of reals and integers), and let $l \in \{0,1\}$ and $\eta < \omega_1$. We define a coding forcing $\code(t, l, \eta)$ over $W$ which codes the tuple $t$ into the branches of $\vec{S}^l$. We illustrate this for $l=1$.

We denote the $\eta$-th coordinate of the generic filter $G^1$ by $g_\eta \subset \omega_1$. Fixing a canonically constructible bijection $\rho: ([\omega_1]^\omega)^L \to \omega_1$, we define a subset $h \subset \omega_1$:
\[ h := \{\rho(g_\eta \cap \alpha) \mid \alpha < \omega_1\}. \]
This $h$ indexes the $\omega$-blocks of $\vec{S}^1$ where $t$ will be coded. We define the block assignment $A \subset \omega_1$ as:
\begin{align*}
A := &\{ \omega\gamma + 2n \mid \gamma \in h, n \notin t \} \cup \\
     &\{ \omega\gamma + 2n + 1 \mid \gamma \in h, n \in t \}.
\end{align*}

Let $X \subset \omega_1$ be chosen such that it canonically codes the set $A$ and the corresponding set of generic branches $\{b_\beta \subset S^1_\beta \mid \beta \in A\}$ added by $\mathbb{P}^0$. Working in $L[X]$, one can decode $t$ by inspecting the $\omega$-block of trees starting at any $\gamma \in h$:
\begin{itemize}
    \item[$(\ast)_1(\gamma, t)$:] $n \in t$ iff $S^1_{\omega\gamma + 2n + 1}$ has an $\omega_1$-branch, and $n \notin t$ iff $S^1_{\omega\gamma + 2n}$ has an $\omega_1$-branch.
\end{itemize}

Applying an argument resembling David's trick \cite{david1982very}, we rewrite the information of $X \subset \omega_1$ as a subset $Y \subset \omega_1$. Any transitive, $\aleph_1$-sized model $N \models \mathsf{ZFC}^-$ containing $X$ will be able to define $\vec{S}^1$ correctly and decode $t$ via $(\ast)_1$. If we code the model $(N, \in)$ as a set $X_N \subset \omega_1$, there will be an $\aleph_1$-sized ordinal $\beta$ such that $L_\beta[X_N]$ correctly believes the decoded model satisfies $(\ast)_1(\gamma, t)$ for every $\gamma \in h$. We can fix a club $C \subset \omega_1$ and a sequence of countable elementary submodels $(M_\alpha \mid \alpha \in C)$ of $L_\beta[X_N]$. 

We let the set $Y \subset \omega_1$ code the pair $(C, X_N)$ such that the odd entries of $Y$ code $X_N$ and the even entries $E(Y)$ code the enumeration $\{c_\alpha \mid \alpha < \omega_1\}$ of $C$ continuously (e.g., $E(Y) \cap [c_\alpha, c_\alpha+\omega)$ codes a well-ordering of type $c_{\alpha+1}$). This structural encoding creates a local version of our decoding property:
\begin{align*}
\sigma_1(t) := {} & \parbox[t]{0.75\textwidth}{For any countable transitive model $M \models \mathsf{ZFC}^- + \text{``}\aleph_1 \text{ exists''}$ such that $\omega_1^M=(\omega_1^L)^M$ and $Y \cap \omega_1^M \in M$, $M$ can construct $L[Y \cap \omega_1^M]$, which in turn sees that there is an $\aleph_1^M$-sized transitive model $\bar{N} \models \mathsf{ZFC}^-$ such that $\bar{N} \models (\ast)_1(\gamma, t)$ for $\aleph_1^M$-many $\gamma$.}
\end{align*}

We define our coding forcing as the almost disjoint coding of this specifically crafted $Y$:
\[ \code(t, 1, \eta) := \mathbb{A}_D(Y). \]
The fundamental effect of this forcing is the addition of a generic real $r$ that acts as a witness to the following $\Pi^1_2$ formula in the parameters $r$ and $t$:
\begin{align*}
\Psi_1(r, t) := {} & \parbox[t]{0.75\textwidth}{For any countable, transitive model $M \models \mathsf{ZFC}^- + \text{``}\aleph_1 \text{ exists''}$ such that $\omega_1^M = (\omega_1^L)^M$ and $r \in M$, $M$ can construct $L[r]^M$, which in turn satisfies that there is a transitive $\mathsf{ZFC}^-$-model $\bar{N}$ of size $\aleph_1^M$ such that $\bar{N} \models (\ast)_1(\gamma, t)$ for an $\aleph_1^M$-sized set of ordinals $\gamma$.}
\end{align*}

When $\Psi_1(r, t)$ holds, we say that $t$ is coded into $\vec{S}^1$. The forcing $\code(t,0,\eta)$ and the corresponding formula $\Psi_0(r, t)$ are defined identically using $\vec{S}^0$. 

\subsection{Allowable Forcings}
To iterate these coding forcings without inadvertent interference, we restrict ourselves to iterations which code reals in an ``injective'' way.

\begin{definition} \label{Def:Allowable}
Let $\alpha < \omega_1$ and let $F \in L$, $F: \alpha \to L$ be a bookkeeping function. A finite support iteration $\mathbb{P} = (\mathbb{P}_\beta \mid \beta < \alpha)$ is called \emph{$0$-allowable (relative to $F$)} if $F$ determines $\mathbb{P}$ inductively as follows:
\begin{itemize}
    \item[] Assume $\beta \ge 0$ and $\mathbb{P}_\beta$ is defined. Let $G_\beta$ be $\mathbb{P}_\beta$-generic over $W$. Assume $F(\beta) = (\dot{t}, \dot{l}, \dot{\eta})$ evaluates to a tuple of parameters $t$, an index $l \in \{0,1\}$, and an ordinal $\eta < \omega_1$.
    \item The freshness condition: If there is a $\gamma < \beta$ where the iteration evaluated $F(\gamma) = (\dot{t}', \dot{l}', \dot{\eta}')$ such that $\dot{\eta}^{\prime G_\gamma} = \eta$ (i.e., the coding area $\eta$ has already been utilized), then $\mathbb{P}(\beta)^{G_\beta}$ is the trivial forcing.
    \item {Coding:} If The freshness condition is met (no such $\gamma < \beta$ exists), then $\mathbb{P}(\beta)^{G_\beta} := \code(t, l, \eta)$.
\end{itemize}
\end{definition}

\begin{definition}
Let $\forceP= ((\forceP_{\beta}, \dot{\forceQ}_{\beta}) \mid \beta < \delta)$ be an allowable forcing. Let $G \subset \forceP$ be a generic filter over $W$. Then
\begin{align*}
C^G:= \{ \eta < \omega_1\mid \exists \beta < \delta \exists \dot{t},&\dot{l}, \dot{\eta} \in W^{\forceP_{\beta}}  \\&( \dot{{\forceQ}}_{\beta})^{G_{\beta}} = \operatorname{Code} (t,l,\eta )\}
\end{align*}
is the set of coding areas of $\forceP$ relative to $G$. 
We also let 
\[ C^{\forceP} := \{\eta < \omega_1 \mid \exists p \in \forceP ( p \Vdash \eta  \in C^{\dot{G}} \}. \]
\end{definition}
It is immediate from the definition that $C^G$ and also $C^{\forceP}$ are always countable sets for every allowable $\forceP$. Next we derive some properties of allowable forcings.

\begin{definition} \label{Def:AllowableOverExtension}
Let $\mathbb{P}$ be an allowable forcing over $W$ of length $\delta_1$, and let $G_{\mathbb{P}}$ be $\mathbb{P}$-generic over $W$. A finite support iteration $\mathbb{Q} = (\mathbb{Q}_\beta \mid \beta < \delta_2)$ in $W[G_{\mathbb{P}}]$ is called \emph{allowable over the extension $W[G_{\mathbb{P}}]$} if there exists a bookkeeping function $F_{\mathbb{Q}}: \delta_2 \to W[G_{\mathbb{P}}]$ such that $\mathbb{Q}(\beta)^{G_{\mathbb{Q}_\beta}}$ is determined exactly as in Definition \ref{Def:Allowable}, but subject to the global freshness condition: there is no prior stage in $\mathbb{Q}$ ($\alpha < \beta$) \emph{and} no stage in the base forcing $\mathbb{P}$ ($\xi < \delta_1$) that utilized the proposed coding area $\eta$.
\end{definition}

\begin{lemma} \label{Lemma:PropertiesOfAllowable}
Allowable forcings exhibit the following properties:
\begin{enumerate}
    \item Let $\mathbb{P}$ be an allowable forcing over $W$ of length $\delta_1$, and suppose $\mathbb{P} \Vdash \text{``}\dot{\mathbb{Q}} \text{ is allowable over } W[G_{\mathbb{P}}] \text{ of length } \delta_2\text{''}$. Then the two-step iteration $\mathbb{P} \ast \dot{\mathbb{Q}}$ is an allowable forcing over $W$ of length $\delta_1 + \delta_2$.
    \item Every allowable forcing over $W$ is ccc, preserves cardinals, and preserves $\mathsf{CH}$.
    \item The product of two allowable forcings $\mathbb{P}$ and $\mathbb{Q}$ can be densely embedded into an allowable forcing provided that their utilized coding areas are disjoint, $C^{\mathbb{P}} \cap C^{\mathbb{Q}} = \emptyset$.
\end{enumerate}
\end{lemma}

Because the definition of an allowable forcing $\mathbb{P}$ depends only on the names of the reals listed by the bookkeeping function and the specific branches through $\vec{S}$ utilized by the coding areas, it can be defined inside a proper inner model.

\begin{lemma}\label{Lemma:DefinabilityInInnerModels}
    Let $\forceP \in W$ be an allowable forcing and let $F: \delta \rightarrow L$ be its bookkeeping. Then there is an uncountable, co-uncountable subset $I \subset \omega_1$ and a countable subset $J \subset \omega_1$ such that
    $\forceP$ can successfully be defined already in an inner model of $W$ of the form
    $L[(G^0 \upharpoonright I)  \times (G^1 \upharpoonright J)]$, where $G^0 \upharpoonright I$ is just the restriction of the generic $G^0$ to coordinates which are in I and $G^1 \upharpoonright J$ being defined similarly.
\end{lemma}
\begin{proof}
    The proof is via induction on the length $\delta$ of the iteration.
    Assume first that $\delta=1$ then we can assume that $\forceP = \operatorname{Code} (x)$ for some real $x \in W$. Recall that $W$ is defined as the generic extension of $L$ via $\forceP^0$ which generically adds branches to each tree in $\vec{S}$ and $\forceP^1= (\prod_{i \omega_1} \mathbb{C} (\omega_1))^L$. The reals in $W= L[G^0 \times G^1]$ are all elements of $L[G^0]$ already as is immediate from Easton's Lemma (see Lemma 15.19 from \cite{Jech}). So in particular there is a countable $I \subset \omega_1$ such that $w \in L[G^0 \upharpoonright I]$ which shows the lemma for one step iterations.

    Now we assume that the lemma is true for allowable forcings of length $\delta$ and our goal is to show that it also must be true for allowable forcings of length $\delta+1$.
    We fix an allowable forcing $\forceP$ of length $\delta+1$ and write $\forceP= \forceP_{\delta} \ast \operatorname{Code} (\dot{w})$ for some $(\forceP^0 \times \forceP^1)\ast \forceP_{\delta}$-name of a real $\dot{w}$. By our induction hypothesis $\forceP_{\delta}$ is already definable in some inner model $L[G^0 \upharpoonright I] [G^1 \upharpoonright J]$. So the name $\dot{w}$ can be written as a $[(\forceP^0 \upharpoonright I) \times (\forceP^1 \upharpoonright J) \ast \forceP_{\delta}] \times [(\forceP^0 \upharpoonright (\omega_1 \setminus I)) \times \forceP^1 \upharpoonright (\omega_1 \setminus J)]$-name. Again by Easton's Lemma, this time applied over the ground model $L[G^1 \upharpoonright J]$, we obtain that the name $\dot{w}$ is in fact (equivalent to) a   $[(\forceP^0 \upharpoonright I) \times (\forceP^1 \upharpoonright J) \ast \forceP_{\delta}] \times (\forceP^0 \upharpoonright (\omega_1 \setminus I))$-name. Note that the forcing $\forceP^0 \upharpoonright (\omega_1 \setminus I)$ is an $\omega_1$-length iteration with finite support which adds a branch through every Suslin tree indexed in the set $\omega_1 \setminus I$ so by standard facts of forcing theory the name $\dot{w}$ is in fact (equivalent to) a $(\forceP^0 \upharpoonright I) \times (\forceP^1 \upharpoonright J) \ast \forceP_{\delta} \times (\forceP^0 \upharpoonright \tilde{I}))$-name, for a countable set $\tilde{I}$. We add $\tilde{I}$ to $I$ and finally note that the coding forcing $\operatorname{Code}(\dot{w})$, as it is ccc will potentially only use a countable set of coding areas $\tilde{J}$. So $\forceP= \forceP_{\delta} \ast \operatorname{Code} (\dot{w})$ can successfully be defined in the inner model
    $L[G^0 \upharpoonright I \cup \tilde{I}][G^1 \upharpoonright J \cup \tilde{J}]$ which shows the successor case.

    The limit case follows immediately from the above via taking countable unions. 
    \end{proof}

\subsection{No Unwanted Codes}

A critical feature of allowable forcings is that they do not accidentally encode unwanted information into the sequences $\vec{S}^0$ or $\vec{S}^1$. When the iteration runs linearly, we must guarantee that the only tuples validating the coding formula $\Psi_l$ are exactly those explicitly targeted by the bookkeeping function.

\begin{lemma}[No Unwanted Codes] \label{Lemma:NoUnwantedCodes}
Let $\mathbb{P} = (\mathbb{P}_\beta \mid \beta < \delta)$ be an allowable forcing over $W$, and let $G \subset \mathbb{P}$ be a generic filter over $W$. For $l \in \{0,1\}$, let $C^l$ be the set of tuples coded by a non-trivial coding factor into $\vec{S}^l$ along the iteration. Then in $W[G]$, if $t$ is a tuple and there exists a real $r$ witnessing $\Psi_l(r, t)$ for $\vec{S}^l$, then $t \in C^l$.
\end{lemma}

\begin{proof}
We outline the proof for $l=1$; the case for $l=0$ is identical. Let $G$ be $\mathbb{P}$-generic over $W$, we work in $W[G]$. Let $g= (g_{\beta} \mid \beta < \delta)$ be the sequence of countable coding areas of $\mathbb{P}$ relative to $G$. We let $\rho: ([\omega_1]^{\omega})^L \rightarrow \omega_1$ be our fixed, constructible bijection and let $h_{\beta}= \{\rho(g_{\beta} \cap \gamma) \mid \gamma < \delta\}$. By the global freshness condition, all utilized coding areas $g_{\beta}$ are distinct coordinates of the Cohen generic $G^1$, ensuring the family $\{h_{\beta} \mid \beta < \delta\}$ forms an almost disjoint family of subsets of $\omega_1$. Thus there is an $\eta < \omega_1$ such that for arbitrary distinct $\beta_1, \beta_2 < \delta$, $h_{\beta_1}\cap h_{\beta_2} \subset \eta$.

We assume for a contradiction that there is a tuple $t \notin C^1$ which satisfies $\Psi_1(r, t)$ for some generic real $r \in W[G]$. As a consequence, $r$ defines an unbounded set $h \subset \omega_1$ such that for every $\gamma \in h$ the following holds true:
$$n \in t \Rightarrow L[r] \models \text{``}S^1_{\omega \gamma + 2n+1} \text{ has an } \omega_1\text{-branch''}$$
and
$$n \notin t \Rightarrow L[r] \models \text{``}S^1_{\omega \gamma + 2n} \text{ has an } \omega_1\text{-branch''}$$

As $t$ is distinct from every tuple $t_\beta \in C^1$, and because $h$ is unbounded, there must be $\aleph_1$-many $\alpha \in h$ with $\alpha > \eta$ and an $n \in \omega$ such that, without loss of generality:
$$L[r] \models \text{``}S^1_{\omega \alpha + 2n+1} \text{ has an } \omega_1\text{-branch''}$$
yet for every $r_{\beta}$ witnessing that $\Psi_1(r_{\beta}, t_{\beta})$ holds true:
$$L[r_{\beta}] \models \text{``}S^1_{\omega \alpha + 2n+1} \text{ does not have an } \omega_1\text{-branch''}$$

Invoking the proof of the definability lemma, there is an uncountable, co-uncountable $I \subset \omega_1$ and a countable $J \subset \omega_1$ such that $\mathbb{P} \in L[G^0 \restriction I][G^1 \restriction J]$ and the set $I$ consists of all the indices of trees determined by the coding areas $G^1 \restriction J$ with the addition of only countably many indices from $\vec{S}$. We denote this countable set of indices with $C$.

We collect the sets:
$$A_{\beta} := \{\omega \gamma +2n \mid \gamma \in h_{\beta}, n \notin t_{\beta}\} \cup \{\omega \gamma + 2n+1 \mid \gamma \in h_{\beta}, n \in t_{\beta}\}$$

As $G \subset \mathbb{P}$ picks at each stage $\beta < \delta$ of its iteration exactly one coding area $g_{\beta}$, the generic extension $L[G^0 \restriction I][G^1 \restriction J][G]$ can be re-written as:
$$L[\{G^0_{\xi} \mid \xi \in \bigcup_{\beta < \delta} A_{\beta}\}][\{G^0_{\xi} \mid \xi \in C\}][(g_{\beta} \mid \beta < \delta)][\{G^0_{\xi} \mid \xi \in I \land \xi \notin (C \cup \bigcup_{\beta <\delta} A_{\beta})\}]$$

In particular, each tree from $\vec{S}^1$ with index not in $C \cup \bigcup_{\beta <\delta} A_{\beta}$ is still Suslin in the inner model $L[\{G^0_{\xi} \mid \exists \beta < \delta (\xi \in A_{\beta})\}][\{G^1_{\xi} \mid \xi \in C\}][(g_{\beta} \mid \beta < \delta)]$ by the independence of the sequence.

It is therefore possible to fix an index $\xi = \omega \alpha + 2n+1 \notin C \cup \bigcup_{\beta <\delta} A_{\beta}$ such that there is a real $r$ with:
$$L[r] \models \text{``}S^1_\xi \text{ has an } \omega_1\text{-branch''}$$
whereas for every $\beta < \delta$:
$$L[r_{\beta}] \models \text{``}S^1_\xi \text{ does not have an } \omega_1\text{-branch''}$$

We claim however that there is no real in $W[G]$ such that $W[G] \models L[r] \models \text{``}S^1_\xi \text{ has an } \omega_1\text{-branch''}$, which will be the desired contradiction.

We show this by pulling the forcing $S^1_\xi$ out of the forcing $\mathbb{P}^0 \times \mathbb{P}^1 \ast \mathbb{P}$ over $L$ which produces $W[G]$. Indeed, if we consider $W[G]=L[\mathbb{P}^0][\mathbb{P}^1][G]$, we can rearrange the generics to $W[G]= L[G'^0 \times G^0_\xi][G^1][G] = L[G'][G^0_\xi]$, where $G'^0 = \prod_{\beta \ne \xi} G^0_{\beta}$ and $G' = G'^0 \times G^1 \ast G$.

Note now that $S^1_\xi$ is still a Suslin tree in $L[G']$ so the forcing $S^1_\xi$ is $\omega$-distributive. This can be seen using the fact that $\vec{S}^0$ and $\vec{S}^1$ are independent. Indeed $S^1_\xi$ will remain Suslin in $L[G'^0][G^1]$, as we can write the universe as $L[G^1][G'^0]$. But then the finite support iteration of the almost disjoint coding forcings is Knaster, hence keeps $S^1_\xi$ Suslin. Consequently $2^{\omega} \cap W[G] = 2^{\omega} \cap L[G']$.

But this implies that:
$$L[G'] \models \neg \exists r \big( L[r] \models \text{``}S^1_\xi \text{ has an } \omega_1\text{-branch''} \big)$$
as the existence of an $\omega_1$-branch through $S^1_\xi$ in the inner model $L[r]$ would imply the existence of such a branch in $L[G']$. Further, as no new reals appear when passing to $W[G]$ we also get:
$$W[G] \models \neg \exists r \big( L[r] \models \text{``}S^1_\xi \text{ has an } \omega_1\text{-branch''} \big)$$
This is the desired contradiction.
\end{proof}

\section{The Full \texorpdfstring{$\omega_1$}{omega1}-Iteration and Diagonalization Scheme}

In this section, we weave the two strands of our construction---the $\mathbf{\Sigma}^1_3$-uniformization and the $\Sigma^1_4$-diagonalization---into a single finite support iteration of length $\omega_1$. 

We list all lightface $\Sigma^1_3$ formulas in two free variables with one real parameter as $(\phi_n(v_0,v_1,p) \mid n \in \omega)$. We also list all lightface $\Sigma^1_4$ formulas in two free variables as $(\varphi_n \mid n \in \omega )$, where \[ \varphi_m(v_0, v_1) \equiv \exists a_0 \forall a_1 \psi_m(v_0, v_1, a_0, a_1)\] and $\psi_m$ is a $\Sigma^1_3$ formula.

We designate a specific lightface $\Sigma^1_4$ formula $\sigma(v_0, v_1)$ which will eventually serve as our counterexample to $\Sigma^1_4$-uniformization:
\[ \sigma(v_0, v_1) \equiv \exists a_0 \forall a_1 \big( (v_0, v_1, a_0, a_1) \text{ is not coded into } \vec{S}^1 \big). \]

To ensure $\sigma$ cannot be uniformized by any $\Sigma^1_4$ set $A_m := \{(x,y) \mid \varphi_m(x,y)\}$, we construct an allowable iteration of length $\omega_1$ which diagonalizes against all possible candidates. We utilize the disjoint sequences of independent Suslin trees: $\vec{S}^0$ is strictly reserved for locking in $\mathbf{\Sigma}^1_3$ witnesses at odd stages, and $\vec{S}^1$ is strictly reserved for the $\Sigma^1_4$ diagonalization scheme at even stages.

\subsection{Weakly Allowable and Suitable Forcings}
To ensure the mathematical validity of our diagonalization---specifically the squaring argument utilized later---we must relax the strict disjointness of coding areas which was part of the definition of allowable forcings. For another technical reason, we also want to add plain Cohen forcings as factors in our altered notion of allowable. The result is the following definition:

\begin{definition}[Weakly Allowable Forcing]
A finite support iteration $\forceQ$ is called \emph{weakly allowable} if it is guided by a bookkeeping function similarly to Definition \ref{Def:AllowableOverExtension} and may naturally incorporate standard Cohen forcings $\mathbb{C}$ at every stage of the iteration. We add one relaxation: we  drop the {global  freshness condition}. A weakly allowable forcing is permitted to reuse the same coding area $\eta$ for multiple reals.
\end{definition}

Just as allowable forcings, weakly allowable forcings preserve $\mathsf{CH}$ and cardinals. We will have to drop the ``no unwanted codes'' lemma \ref{Lemma:NoUnwantedCodes}, but we keep a weaker version of this which is sufficient for our needs:
\begin{lemma}
    Suppose that $\forceP= (\forceP_{\beta} \mid \beta < \delta)$ is a weakly allowable forcing. Let $G$ be $\forceP$ generic and assume that \[D^{G}:=\{ \eta < \omega_1 \mid \text{the coding area with index $\eta$ has been used at least twice by $\forceP^G$}  \} \]
    Let $E^G=\{x \in \omega^{\omega} \mid x \text{ is coded at a coding area $\eta \notin D^G$} \}$, then
    there are no unwanted reals in $E^G$.
\end{lemma}
Its proof is exactly as the proof of lemma \ref{Lemma:NoUnwantedCodes}, modulo the set $D^{G}$, so we skip it. In our iteration which shall prove the main theorem, the set $D^G$ will be a countable set, hence we can code it with a real. This real represents the coding areas we need to discard, when defining $\bf{\Sigma}^1_3$-uniformizing functions.

We need to keep track  of the elements of our $\Sigma^1_4$-set which eventually will witness the failure of $\Sigma^1_4$-uniformiation. For this we introduce the notion of a restricition set.  A restriction set $E$ is a collection of tuples of the form $(x, y, a_0) \in (\omega^\omega)^3$. The role of $E$ is to act as a strict blacklist: it records the putative values and challenges for which we must absolutely prevent any future coding into the $\vec{S}^1$ sequence. 

\begin{definition}[Suitable Forcing] \label{Def:Suitable}
Let $M$ be a generic extension of $W$, and let $E \in M$ be a restriction set. A finite support iteration $\forceQ = (\forceQ_\xi \mid \xi < \delta)$ over $M$ is called \emph{suitable with respect to $E$} if:
\begin{enumerate}
    \item $\forceQ$ is a weakly allowable forcing.
    \item For every stage $\xi < \delta$, every $\forceQ_\xi$-generic filter $H_\xi$ over $M$, every tuple $(x, y, a_0) \in E$, every real $a_1 \in M[H_\xi]$, and every ordinal $\eta < \omega_1$, the specific coding forcing into the $\vec{S}^1$ sequence,
    \[ \code(x, y, a_0, a_1, 1, \eta), \]
    is never utilized as the iterand $\forceQ(\xi)^{H_\xi}$.
\end{enumerate}
\end{definition}

When working within our overarching construction at stage $\beta < \omega_1$, we will have generated a specific restriction set $E_\beta \in W[G_\beta]$. We will succinctly refer to a forcing $\forceQ \in W[G_\beta]$ as being \emph{$\beta$-suitable} if it is suitable with respect to $E_\beta$. 

The restriction set $E_\beta$ is the critical ledger for our diagonalization. By explicitly boycotting the coding of any $a_1$ extending a tuple $(x, y, a_0) \in E_\beta$, a $\beta$-suitable forcing mathematically guarantees that the existential statement $\exists a_0 \forall a_1 \big( (x, y, a_0, a_1) \text{ is not coded into } \vec{S}^1 \big)$ remains true for the pair $(x,y)$ in any subsequent suitable generic extension. 

\begin{lemma}[Product of Suitable Forcings] \label{Lemma:ProductSuitable}
Let $\forceQ_0$ and $\forceQ_1$ be two $\beta$-suitable forcings over $W[G_\beta]$ with associated internal restriction sets $E_{\forceQ_0}$ and $E_{\forceQ_1}$. Suppose further that $\forceQ_0$ never uses a coding factor into $\vec{S}^1$ for any tuple restricted by $E_{\forceQ_1}$, and symmetrically, $\forceQ_1$ never uses a coding factor for any tuple restricted by $E_{\forceQ_0}$. Then their finite support product $\forceQ_0 \times \forceQ_1$ can be densely embedded into a weakly allowable forcing that is suitable with respect to the combined restriction set $E_\beta \cup E_{\forceQ_0} \cup E_{\forceQ_1}$.
\end{lemma}
\begin{proof}
Because weakly allowable forcings do not require The freshness condition, the interleaving of the factors of $\forceQ_0$ and $\forceQ_1$ remains a valid weakly allowable iteration. Since $\forceQ_0$ is $\beta$-suitable, it strictly satisfies the conditions of Definition \ref{Def:Suitable} and avoids coding any tuples in $E_\beta \cup E_{\forceQ_0}$. By our non-interference hypothesis, it also avoids coding any tuple restricted by $E_{\forceQ_1}$. Symmetrically, $\forceQ_1$ avoids coding tuples in $E_\beta \cup E_{\forceQ_1} \cup E_{\forceQ_0}$. Therefore, their finite support product  avoids coding any tuples in the union $E_\beta \cup E_{\forceQ_0} \cup E_{\forceQ_1}$, preserving suitability over the combined restriction set.
\end{proof}

\subsection{The Alternating Iteration Scheme}
We fix a recursive partition $\langle P_m \mid m \in \omega \rangle$ of the Baire space $\omega^\omega$ into $\omega$-many pairwise disjoint sets. A real $x$ is assigned to the integer $m$ if and only if $x \in P_m$. 

Let $F: \omega_1 \to H(\omega_1)$ be a bookkeeping function such that every tuple in $H(\omega_1)$ appears cofinally often on both the even and odd ordinals. We fix a canonical wellordering $<_W$ of the ground model $W$. We proceed by induction on $\beta < \omega_1$. Let $E_0 = \emptyset$, and for limit stages $\lambda$, $E_\lambda = \bigcup_{\xi < \lambda} E_\xi$.  We evaluate $F(\beta)$ using the generic filter $G_\beta$ and proceed according to the parity of $\beta$:

\begin{enumerate}
    \item \textbf{Odd Stages (Towards $\mathbf{\Sigma}^1_3$-Uniformization):}
    If $\beta$ is odd, we interpret $F(\beta)$ as a tuple $(\dot{x}, \dot{p}, \dot{m}, \dot{\eta})$. We drop the dots on the variables to denote the evaluations of the names using $G_{\beta}$. If there is a real $y$ such that $W[G_\beta] \models \exists a_0 \phi_m(x, y,a_0,p)$, where $\phi_m$ is the $m$-th $\Pi^1_2$-formula, we let $y^*$ be such a real  which has the  $<_W$-least name. We force with $\code(x, y^*,a_0,p, m, 0, \eta)$ for the least $\eta$ where The freshness condition holds to encode this specific real $y^*$ into the $\vec{S}^0$ sequence. We update our restriction set $E_{\beta+1}$ to forbid any future coding of a different real $y' \neq y^*$ for the tuple $(x, y',p, m)$. If no such $y^*$ exists, we force trivially and set $E_{\beta+1} := E_\beta$.

    \item \textbf{Even Stages (Towards Failure of $\Sigma^1_4$-Uniformization):}
    If $\beta$ is even, we interpret $F(\beta)$ as a tuple $(\dot{x}, \dot{a}_0, \dot{m}, \dot{l}, \dot{\eta})$. Again, we drop the dots on the variables to denote the evaluations of the names using $G_{\beta}$. If $x \notin P_m$, we force trivially and set $E_{\beta+1} := E_\beta$. 
    
    If $x \in P_m$, our action depends on whether elements for the $x$-section of $A$ computed by our distinguished $\Sigma^1_4$-formula $\sigma(v_0,v_1)$ have already been defined earlier in our iteration, combined with the indicator bit $l \in \{0, 1\}$:
    \begin{itemize}
        \item \textbf{Option 1 (Creation: $l = 0$):} If elements for the $x$-section of $A$ have \emph{not} yet been defined at any earlier stage of $\forceP$, we generically adjoin $\eta$-many such values via the finite support product of Cohen forcings $\forceP(\beta)^{G_\beta} := (\prod_{i<\eta} \mathbb{C}_i) \times (\prod_{i<\eta} \mathbb{C}_i)$. Let $\{c_i(x,m) \mid i < \eta\}$ and $\{a_0^i \mid i < \eta\}$ denote the mutually generic Cohen reals added. We permanently forbid any coding into $\vec{S}^1$ involving these pairs :
        \[ E_{\beta+1} := E_\beta \cup \big\{ (x, c_i(x,m), a_0^i) \mid i < \eta \big\}. \]
        Note that this has the consequence that the $a_0^i$ witness that $(x,c_i(x,m))$ are elements of the  set $A$ defined by our $\sigma(v_0,v_1)$ in all outer models obtained with a $\beta+1$-suitable forcing.

        \item \textbf{Option 2 (Diagonalization: $l = 1$):} If elements of $A$ on the $x$-section \emph{were} already defined at an earlier stage $\gamma < \beta$, we consider the given real $a_0$ and the $m$-th $\Sigma^1_4$-formula $\varphi_m = \exists d_0 \forall d_1 \psi_m (v_0,v_1,d_0,d_1)$. We check if there exists a real $y$ such that the tuple $(x, y, a_0) \in E_\beta$. If so, we ask if there is a \emph{$\beta$-suitable} forcing $\forceQ \in W[G_\beta]$ that adds a real $a_1$ satisfying:
        \[ \forceQ \forces \neg\psi_m(x, y, a_0, \dot{a}_1). \]
        If such a $\forceQ$ exists, we let $\forceQ^*$ be the $<_W$-least such forcing, and set $\forceP(\beta)^{G_\beta} := \forceQ^*$. This actively forces $(x, y,a_0)$ out of $\forall d_1 \psi_m(v_0,v_1,v_2,d_1)$. We set $E_{\beta+1} := E_\beta$. {If no such forcing exists, we say that  we got stuck.}
    \end{itemize}
\end{enumerate}

\subsection{The ``Stuck'' Condition and Squaring the Forcing}
The iteration proceeds linearly unless we reach a stage $\beta$ with $F(\beta)= (\dot{x},\dot{a}_0,\dot{m},\dot{l}, \dot{\eta})$ where case 2, option 2 applies and we can no longer force some triple $(x, y,a_0)$ out of $\forall d_1 \psi_m(v_0,v_1,v_2,d_1)$, in other words we got stuck in our attempt to force $(x,y_0)$ out of $\exists d_0 \forall d_1 \psi_m (v_0,v_1,d_0,d_1)$.

\begin{definition}[The Stuck Condition]
We say the iteration is \emph{stuck} at stage $\beta$ if $F(\beta)$ evaluates to $(x, a_0, m, 1, \eta)$ with $x \in P_m$, there exists a real $y$ such that $(x, y, a_0) \in E_\beta$, and for \emph{every} $\beta$-suitable forcing $\mathbb{Q} \in W[G_\beta]$, we have:
\[ \Vdash_{\mathbb{Q}} \forall \dot{a}_1 \psi_m(x, y, a_0, \dot{a}_1). \]
\end{definition}

If the iteration gets stuck at stage $\beta$ for a real $y$, we must alter the ground model. Let $\gamma < \beta$ be the stage where the elements for the $x$-section of $A$  were generically added via $\mathbb{P}(\gamma)$. The intermediate iteration factor $\mathbb{P}_{\gamma, \beta} \cong \mathbb{P}(\gamma) \ast \mathbb{P}_{\gamma+1, \beta}$ is a $\gamma$-suitable forcing, and it has the property that:
\begin{align*}
   \mathbb{P}(\gamma) \Vdash \exists \dot{a}_0 \Big(& \dot{a}_0 \text{ is a } \mathbb{P}_{\gamma+1, \beta}\text{-name} \land \\& \forceP_{\gamma+1,\beta} \Vdash ``\forall \mathbb{Q} \text{ } \beta\text{-suitable, } \mathbb{Q} \Vdash \forall a_1 \psi_m(x, \dot{y}, \dot{a}_0, a_1)" \Big). 
\end{align*}

\begin{lemma}[The Squaring Argument]
If the iteration is stuck at stage $\beta$, then, working over $W[G_{\gamma}]$, we can replace the factor $\mathbb{P}_{\gamma, \beta}$ with its product square $\mathbb{P}_{\gamma, \beta} \times \mathbb{P}_{\gamma, \beta} $ to create a new model in which the $x$-section of $A_m$ contains at least two distinct elements.
\end{lemma}
\begin{proof}
We work over the universe $W[G_\gamma]$ and force with $\mathbb{P}_{\gamma, \beta}^{(2)} \cong (\mathbb{P}(\gamma) \ast \mathbb{P}_{\gamma+1, \beta}) \times (\mathbb{P}(\gamma) \ast \mathbb{P}_{\gamma+1, \beta})$. Let $G_{\gamma, \beta}^{(2)}=G_{\gamma,\beta}^0 \times G^1_{\gamma,\beta}$ be the generic filter. Let $E_{\gamma,\beta}^i$, for $i \in 2$ denote the two restriction sets we obtained when arriving at the model $W[G_{\gamma}][G_{\gamma,\beta}^0]$ and $W[G_{\gamma}][G_{\gamma,\beta}^1]$ respectively.

Crucially, by Lemma \ref{Lemma:ProductSuitable} and as neither $W[G_{\gamma}] [G^0_{\gamma,\beta}]$ will  code any triples of reals in $E^1_{\beta}$, nor $W[G_{\gamma}] [G^1_{\gamma,\beta}]$ will  code any triples of reals in $E^0_{\beta}$ its product square $\mathbb{P}_{\gamma, \beta}^{(2)}$ is also a  suitable forcing with respect to the combined restriction set $E^0_{\beta} \cup E^1_{\beta}$. The generic $G_{\gamma, \beta}^{(2)}$ yields two mutually generic evaluations of the real with name $\dot{y}$, denoted $y^0$ and $y^1$, and provides two evaluations of the name $\dot{a}_0$, denoted $a_0^0$ and $a_0^1$, stemming from the left and right coordinates of the generic filter.

By the forcing theorem applied in $W[G_\gamma]$, and because the stuck condition quantifies over all suitable forcings with restriction set $E^0_{\beta}$ and $E^1_{\beta}$ respectively the new model $W[G_\gamma][G_{\gamma, \beta}^{(2)}]$ satisfies for each $j \in \{0, 1\}$:
\[ \exists a_0^j \forall \mathbb{Q} \big( \mathbb{Q} \text{ is suitable with restriction } E^{0}_{\beta} \cup E^1_{\beta} \implies \mathbb{Q} \Vdash \forall a_1 \psi_m(x, y^j, a_0^j, a_1) \big). \]
Because $y^0 \neq y^1$, in any further suitable extension with restriction sets extending $E^0_{\beta} \cup E^1_{\beta}$, both $(x, y^0)$ and $(x, y^1)$ will satisfy $\exists a_0 \forall a_1 \psi_m$. Thus, $y^0, y^1 \in (A_m)_x$. Since the $x$-section of $A_m$ contains at least two distinct elements, $A_m$ is not the graph of a function and cannot uniformize $\sigma$.
\end{proof}

To summarize, if we arrive in our linear iteration at a stage $\beta$, where we got stuck, we proceed in squaring the forcing $\forceP_{\gamma,\beta}$, where $\gamma$ is the least stage where values for $\sigma$ at the $x$-section are introduced. Consequently,  the squaring operation permanently neutralizes the set $A_m$  as the graph of a  potential uniformizing function of the $\Sigma^1_4$-set associated with $\sigma(v_0,v_1)$. After a squaring occurred, we start our linear iteration again, until we get stuck in which we just repeat the squaring step and so on. This way, we will iterate $\omega_1$-many steps using our bookkeeping $F$ we fixed in advance. 
Since there are only countably many lightface $\Sigma^1_4$ formulas (and hence countably many indices $m \in \omega$), this rollback and squaring process can occur at most $\omega$ times. Let $\beta^* < \omega_1$ be the supremum of all stages where a squaring operation occurs. Because it is a countable supremum of countable ordinals, $\beta^* < \omega_1$. After stage $\beta^*$, the actual iteration $\mathbb{P}_{\omega_1}$ completely stabilizes and runs linearly as a suitable forcing up to $\omega_1$. Let $G_{\omega_1}$ be generic over $W$ and let $W[G_{\omega_1}]$ denote the final universe. This will be the universe where $\Sigma^1_4$-uniformization fails and $\boldsymbol{\Sigma}^1_3$-uniformization holds as we shall show now.

\section{Proof of the Main Theorem}

\begin{theorem} \label{MainTheorem}
In the generic extension $W[G_{\omega_1}]$, the boldface $\mathbf{\Sigma}^1_3$-uniformization property holds, yet the lightface $\Sigma^1_4$-uniformization property fails.
\end{theorem}

\begin{proof}
We first verify the global failure of lightface $\Sigma^1_4$-uniformization. Consider our designated lightface $\Sigma^1_4$ relation $\sigma(v_0, v_1)$. Let $A_m$ be an arbitrary lightface $\Sigma^1_4$ relation. Because the iteration stabilized, exactly one of two outcomes occurred for the index $m$:
\begin{enumerate}
    \item \textbf{The iteration completed without getting stuck for $m$:} For every $x \in P_m$, reals $c_i(x,m)$ were generically added exactly once, and their ``not being coded'' status was permanently protected by the restriction sets $E_\beta$. Over the linear stages, we successfully forced $(x, c_i(x,m)) \notin A_m$. Thus, $(x, c_i(x,m)) \in \sigma_x \setminus (A_m)_x$, and $A_m$ fails to uniformize $\sigma$.
    \item \textbf{The iteration got stuck for $m$:} The Squaring Argument was invoked prior to $\beta^*$. The relation $A_m$ was forced to contain both $(x, y^0)$ and $(x, y^1)$. Because $y^0 \neq y^1$, the relation $A_m$ is not the graph of a function, and thus cannot uniformize $\sigma$.
\end{enumerate}
In all cases, no lightface $\Sigma^1_4$ relation can uniformize $\sigma$.

We now verify the truth of boldface $\mathbf{\Sigma}^1_3$-uniformization. Let $A_m(p)$ be a $\mathbf{\Sigma}^1_3$ relation in the plane defined by the parameter $p \in \omega^\omega \cap W[G_{\omega_1}]$. 

Because of the rollbacks that occurred prior to stage $\beta^*$, any real $x$ processed during a squared interval $[\gamma, \beta)$ may have had multiple distinct witnesses $y^i$ coded into the $\vec{S}^0$ sequence along the different branches of the product forcing. Therefore, the coding relation
\[ U_m(x, y,p) \equiv \exists a_0 (x,y,a_0,p, m) \text{ is coded into } \vec{S}^0 \]
might not be single-valued for reals processed before $\beta^*$. 

However, the set of all reals added or processed before stage $\beta^*$ is countable in $W[G_{\omega_1}]$. Let $C \subset \omega^\omega$ be this countable set. Because $C$ is countable, we can fix a single real parameter $R^* \in W[G_{\omega_1}]$ that explicitly encodes a chosen uniformizing witness $z_x$ for every $x \in C \cap pA_m(p)$. 

For any real $x \notin C$, $x$ was processed at some odd stage $\alpha > \beta^*$. Because the iteration is strictly linear after $\beta^*$, exactly one $<_W$-least witness $z^*$ was selected and locked into $\vec{S}^0$. Thus, for all $x \notin C$, the relation $U_m(x, y,p)$ is strictly single-valued.

We can now define our uniformizing function $f_{m,p}(x)$ using the parameter $R^*$:
\[
f_{m,p}(x) = z \iff 
\begin{cases} 
R^* \text{ specifies } z \text{ as the witness for } x, & \text{if } x \in C \\
U_m(x,y,p) \text{ holds}, & \text{if } x \notin C 
\end{cases}
\]
This patched relation is clearly $\mathbf{\Sigma}^1_3(R^*, p)$, it has the same domain as $A_m(p)$, and it is strictly single-valued. Thus, boldface $\mathbf{\Sigma}^1_3$-uniformization holds. This completes the proof.
\end{proof}

\section{Separating $\Sigma^1_3$- and $\Sigma^1_4$ uniformization}

We next present an argument which produces a universe where $\Sigma^1_3$-uniformization holds, yet no projective level of complexity $\ge 4$ has the uniformization property. The argument is a generalization of an old result of A. L\'evy to inner models with large cardinals. We shall force over the inner model $\Lsharp$ and we proceed to introduce the properties of $\Lsharp$ which we need in our proof.

\subsection{Some useful properties of $\Lsharp$}

Recall that $\Lsharp$ denotes the least inner model which is closed under $X \mapsto X^{\#}$ for every set $X$. It is well-known that $ L^{\#}$ can be arranged as a mouse $M$ (where every extender added is a sharp), and for every level $M_\eta$ of $M$ and $\Sigma_1$-elementary substructure $\overline{M}$ of $M_\eta$ collapses to some $M_{\eta'}$ for some $\eta' \le \eta$.

We shall make crucial use of the fact that countable initial segments of $\Lsharp$ obey a $\Pi^1_2$ definition. Say a ``simple $x$-mouse'' is a mouse $M$ of the form in the hierarchy of $\Lsharp_x$: it has $x$ at the bottom, and for every segment of $M$ which has an active extender $F$, the extender is a ``sharp'', meaning that letting $\kappa = \text{cr}(F)$, then there is a bound $\lambda < \kappa$ such that $M$ has no extenders with index strictly between $\lambda$ and $\text{OR}^M$. (As usual, every proper segment of $M$ must be fully sound.) It is a $\Pi^1_2(x, z)$ statement to say ``$z$ codes a simple $x$-mouse'' (and soundness is also only roughly first-order over $M$, so it's also $\Pi^1_2$ to say ``$z$ codes a sound simple $x$-mouse''). Every sound simple $x$-mouse in $\Lsharp_x$ is in fact a proper segment of $\Lsharp_x$. We make heavy use of 

\begin{theorem}
$L^{\#}$ is two-step $\Sigma^1_3$ generically absolute.
\end{theorem}

\begin{proof}
Let $G$ be set generic over $L^{\#}$.
Given a real $x$, let $\Lsharp_x$ denote the minimal transitive proper class containing $x$ and closed under sharps (for sets).

\begin{claim}
Let $x$ be a real in $L^{\#}[G]$ (so $L^{\#}[x]$ is a sub-generic extension). Then $L^{\#}[x] = L^{\#}_x$.
\end{claim}

\begin{proof}
We have $L^{\#}[x] \subseteq L^{\#}_x$ because $L^{\#} \subseteq L^{\#}_x$ by the minimality of $\Lsharp$ under closure under \#'s, and of course because $x \in \Lsharp_x$.
And $\Lsharp_x \subseteq \Lsharp[x]$ by the minimality of $\Lsharp_x$ (as $x \in \Lsharp[x]$ and $\Lsharp[x]$ is closed under sharps).
\end{proof}

\begin{claim}
$\Lsharp[x] \prec_{\Sigma^1_3} \Lsharp[G]$.
\end{claim}

\begin{proof}
Suppose that $\Lsharp [G] \models \phi(y)$ where $\phi$ is $\Sigma^1_3$ and $y \in R \cap \Lsharp [x]$. Now $\Lsharp [G]$ is a forcing extension of $\Lsharp [x]$, by say forcing $\mathbb{P} \in \Lsharp [x]$.
Work in $\Lsharp [x]$. Let $\kappa$ be large enough that
$$x, y, \mathbb{P} \in N = (\Lsharp |_{\kappa})[x];$$
and then let $\lambda > \kappa$ be such that $N^\# \in (\Lsharp |_{\lambda})[x]$. Let $X \prec (\Lsharp|_{\lambda})[x]$ with $\mathbb{P}, x, y, N, N^\# \in X$ and $X$ countable. Let $M$ be the transitive collapse of $X$ and write $\overline{\mathbb{P}}$ for the transitive collapse image of $\mathbb{P}$, etc. Let $g$ be $M$-generic for $\overline{\mathbb{P}}$ and such that $M[g] \models \phi(y)$. Write $\phi(y) = \exists z \psi(y, z)$ with $\psi$ being $\Pi^1_2$. Let $z \in M[g]$ be such that $M[g] \models \psi(y, z)$. Then $\psi(y, z)$ is really true, because the collapse image of $N^\#$ is the true $\overline{N}^\#$, and this extends to give the true $(\overline{N}[g])^\#$, and $(\overline{N}[g])^\#$ thinks $\psi(y, z)$ is true, so it really is true. But $z \in \Lsharp[x]$, so $\Lsharp[x] \models \phi(y)$.
\end{proof}

\end{proof}

We can now prove that $\Lsharp[G]$ satisfies ${\Sigma}^1_3$-uniformization. 
\begin{theorem}
In any set generic extension of $L^{\#}$, denoted by $L^{\#}[G]$, $\Sigma^1_3$-uniformization holds.
\end{theorem}
\begin{proof}

Let $\phi(u, v)$ be a ${\Sigma}^1_3$ formula in two variables. Recall that there is the $\Pi^1_2$-formula asserting that ``$M$ is sound, simple $x$-mouse''. This formula can be used to obtain a ${\Sigma}^1_3(x)$-good wellorder of $R \cap \Lsharp_x$, uniformly in $x$, and this leads as usual to ${\Sigma}^1_3$-uniformization inside $\Lsharp_x$, but then by the $\Sigma^1_3$-elementarity of Subclaim 2, the same uniformization works in $\Lsharp[G]$.

More explicitly, say $\phi(x, y)$ says $\exists z \psi(x, y, z)$ where $\psi$ is $\Pi^1_2$. Working in $\Lsharp[G]$, we can then uniformize
$$\{(x, y) \mid \Lsharp[G] \models \phi(x, y)\}$$
with the formula $\phi'$, where $\phi'(x, y)$ says ``there is a sound simple $x$-mouse $M$ and $z \in R$ such that:
\begin{itemize}
    \item $y, z \in M$ ($x$ is automatically in $M$),
    \item $M \models \psi(x, y, z)$ and $\langle x, y, z \rangle^\# \in M$,
    \item there is no proper initial segment $M_0$ of $M$ and $y_0, z_0$ such that $(M_0, y_0, z_0)$ have the properties above,
    \item $M$ thinks that $(y, z)$ is the least pair $(y_0, z_0)$ with the properties above (in its canonical ordering).
\end{itemize}
By the $\Pi^1_2$-definability of ``sound simple $x$-mouse'', $\phi'$ is ${\Sigma}^1_3$. And $\phi'$ uniformizes $\phi$. For (i) the minimality of $(M, y, z)$ above (and that all sound simple mice $N, N'$, have either $N \unlhd N'$ or $N' \unlhd N$) guarantees
$$\phi'(x, y) \land \phi'(x, y_0) \implies y = y_0;$$
and (ii)
$$\phi'(x, y) \implies \phi(x, y);$$
as witnessed by the $z$ above, because
$$M \models \psi(x, y, z) \text{ and } \langle x, y, z \rangle^\# \in M;$$
and (iii) if $\phi(x, y_0)$ then by Subclaim 2, $\Lsharp_x \models \exists y \phi(x, y)$, so we get some $(y_{1}, z_{1}) \in \Lsharp_x$ such that $\psi(x, y_{1}, z_{1})$, but then working in $\Lsharp_x$ we can minimize and get the $(M, y, z)$ as above, so $\phi'(x, y)$.
\end{proof}

\subsection{Separating $\Sigma^1_3$-uniformization from $\Sigma^1_4$-uniformization}
We shall prove the second theorem of this article now.
Using $\Lsharp$ as our ground model we force with adding $\omega_1$-many side by side Cohen forcings with finite support. We let $\forceP = \prod_{i < \omega_1} \mathbb{C}$, and let $g$ denote a $\forceP$-generic filter over $\Lsharp$.
We consider the following formula $\phi(x,y) \Leftrightarrow y \notin \Lsharp [x]$. Note that the complexity of $\phi(x,y)$ is $\Pi^1_3$, as $y \notin \Lsharp [x]$ can be stated as 
\[\forall \alpha < \omega_1 (y \notin \Lsharp_{\alpha} [x] ) \]
and the set of $\Lsharp_{\alpha} [x]$ are the set of reals $z$ coding a simple $x$ mouse which is a $\Pi^1_2 (x)$-definable set of reals. So $y \notin \Lsharp [x]$ can be written as $\forall z ( z $ codes a simple  $ x\text{-mouse} \rightarrow y \notin z)$
which is $\Pi^1_3 (x)$. By the last section, $\Lsharp [g]$ will satisfy $\Sigma^1_3$-uniformization. What is left is to argue that $\Sigma^1_4$-uniformization fails there. This is just a repetition of A. L\'{e}vy's argument (see \cite{levy1965definability}). The set defined by 
$\phi(x,y)$ in fact can not be uniformized by any ordinal definable function. 

\section{Questions}
The most natural follow up questions is of course whether we can get the separation of $\Sigma^1_3$ and $\Sigma^1_4$ uniformization without large cardinals. A first idea would be to twist the current proof and add another coding forcing to the iteration which makes the real $R^*$ from the proof of theorem \ref{MainTheorem} itself $\Sigma^1_3$-definable. A similar, convoluted idea was used already in \cite{Ho1} and \cite{hoffelner2025forcinguppersigmauniformizationpresence}. 
This idea here can not work however. For assume such an additional coding could be done, then it can also be applied to the similar problem where we want to force a universe where a good $\Sigma^1_3$-definable wellorder is forced in the presence of a failure of $\Sigma^1_4$ uniformization. Indeed we could just replace the stages in the proof of theorem \ref{MainTheorem} where we work towards $\boldsymbol{\Sigma}^1_3$ uniformization with stages where we work towards a good $\Sigma^1_3$-wellorder. Thus if a $\Sigma^1_3$-definition of $R^*$ is possible, we could simultaneously force $\Sigma^1_4$ uniformization and its failure which is nonsense. Consequently the current proof will have to changed substantially if we want to succeed in answering the natural follow up.
\begin{question}
    Assuming just $Con(\mathsf{ZFC})$, can we produce a universe where $\Sigma^1_3$ uniformization holds and $\Sigma^1_4$ uniformization fails.
    \end{question}
There are techniques now which can force $\Sigma^1_n$-uniformization. Thus the following question seems to be in reach:
\begin{question}
    Fix an integer $n \in \omega$. Are there universes where $\Sigma^1_m$-uniformization holds for each $m \le n$, and $\Sigma^1_m$-uniformization fails for each $m \ge n$.
\end{question}

\section{Acknowledgment}
The author's research was funded in whole by the Austrian Science Fund (FWF) Grant-DOI 10.55776/P37228. For the purpose of open access, the author has applied a CC BY public copyright license to any Author Accepted Manuscript version arising from this submission. He thanks F. Schlutzenberg for discussions on a related topic.
\bibliographystyle{plain}
\bibliography{references}

\end{document}